\documentclass[12pt]{article}
\usepackage{amscd}     % nice arrows
\usepackage{amsmath}   % math stuff
\usepackage{amssymb}   % special symbols for math
\usepackage{amsthm}    % gives us the \newtheorem
\usepackage{hyperref}  % creates hyper referencing
\usepackage{bookmark}  % creates automatic bookmarks in the pdf by sections
\usepackage{titlesec}
\usepackage{mathrsfs} % nice caligraphic fonts
\usepackage{bbm}
\usepackage{mathtools} % all sorts of math widgets
\usepackage{verbatim}  % if you want latex to read it exactly as you write it use verbatim environment
\usepackage{graphicx}  % standard for graphicx includes
\usepackage{color}     % enables font coloring
\usepackage{soul}
\usepackage{scalerel,stackengine}

\usepackage[skip=2pt,font={small, it}]{caption} % enables caption customizations

\usepackage{algorithm} % Algorithms environment
\usepackage{algorithmicx}
\usepackage{multirow}

\usepackage{algpseudocode}
\usepackage[normalem]{ulem}

\usepackage[UKenglish]{babel}
\usepackage{authblk}

\usepackage{geometry}
\usepackage{tikz}
	\usetikzlibrary{shapes,arrows}
	\usetikzlibrary{positioning}
	\usetikzlibrary{calc,shapes, positioning}
	\usetikzlibrary{patterns}
	
\DeclareMathAlphabet{\mathpzc}{OT1}{pzc}{m}{it}

\newtheorem{definition}{Definition}

    \newtheorem{Theorem}{Theorem}[section]
\newtheorem{Lemma}[Theorem]{Lemma}

\theoremstyle{remark}

\newcommand{\RR}{\mathbb{R}}

\newcommand{\EE}{\mathbb{E}}

\newcommand{\OO}{\mathcal{O}}

\newcommand{\C}{\mathcal{C}}

\newcommand{\XX}{\mathscr{X}}
\newcommand{\JJ}{\mathcal{J}}
\newcommand{\ZZZ}{\mathscr{Z}}
\newcommand{\YY}{\mathscr{Y}}
\newcommand{\LLL}{\mathscr{L}}

\newcommand{\eps}{\varepsilon}
\newcommand{\defeq}{\stackrel{\bigtriangleup}{=}}

\newcommand{\abs}[1]{\left\vert #1 \right\vert}
\newcommand{\norm}[1]{\left\Vert #1 \right\Vert}

\DeclareMathOperator*{\argmin}{\arg\!\min}

\DeclareMathOperator{\cov}{cov}

\stackMath
\newcommand\reallywidehat[1]{%
\savestack{\tmpbox}{\stretchto{%
  \scaleto{%
    \scalerel*[\widthof{\ensuremath{#1}}]{\kern.1pt\mathchar"0362\kern.1pt}%
    {\rule{0ex}{\textheight}}%WIDTH-LIMITED CIRCUMFLEX
  }{\textheight}% 
}{2.4ex}}%
\stackon[-6.9pt]{#1}{\tmpbox}%
}
\newrobustcmd*{\mysquare}[1]{\tikz{\filldraw[draw=#1,fill=#1] (0,0)
rectangle (0.2cm,0.2cm);}}

\newrobustcmd*{\mycircle}[1]{\tikz{\filldraw[draw=#1,fill=#1] (0,0) circle [radius=0.1cm];}}

\usepackage{enumitem}
\setlistdepth{9}

\newlist{myEnum}{enumerate}{9}
\setlist[myEnum,1]{label=\arabic*.}
\setlist[myEnum,2]{label=(\alph*)}
\setlist[myEnum,3]{label=\roman*.}
\setlist[myEnum,4]{label=\Alph*.}
\setlist[myEnum,5]{label=(\arabic*)}
\setlist[myEnum,6]{label=(\Roman*)}
\setlist[myEnum,7]{label=(\Alph*)}
\setlist[myEnum,8]{label=(\roman*)}
\setlist[myEnum,9]{label=(\arabic*)}

\definecolor{myPurple}{rgb}{0.4940, 0.1840, 0.5560}
\definecolor{myGreen}{rgb}{0, 0.5, 0}

\allowdisplaybreaks

\graphicspath{{./Figures/}}

\begin{document}
\title{Convergence rates of vector-valued local polynomial regression}
\author[1]{Yariv Aizenbud\footnote{authors contributed equally}}
\author[2]{Barak Sober$^*$}
\affil[1]{Department of Mathematics, Yale University}
\affil[2]{Department of Mathematics and Rhodes Information Initiative, Duke University}

\maketitle
\begin{abstract}
Non-parametric estimation of functions as well as their derivatives by means of local-polynomial regression is a subject that was studied in the literature since the late 1970's. 
Given a set of noisy samples of a $\C^k$ smooth function, we perform a local polynomial fit, and by taking its $m$-th derivative we obtain an estimate for the $m$-th function derivative.
The known optimal rates of convergence for this problem for a $k$-times smooth function $f:\mathbb{R}^d \to \mathbb{R}$ are $n^{-\frac{k-m}{2k + d}}$.
However in modern applications it is often the case that we have to estimate a function operating to $\RR^D$, for $D \gg d$ extremely large.
In this work, we prove that these same rates of convergence are also achievable by local-polynomial regression in case of a high dimensional target, given some assumptions on the noise distribution.
%In this work, we prove that these same rates of convergence are also achievable by local-polynomial regression in case of a high dimensional target, given that the noise in the sample has a covariance matrix with operator norm bounded by $\OO(\sqrt{1/D})$.
%Noise models that respect this assumption include uniform distribution on a sphere or ball of radius $\sigma$, or normal distribution with expected value of the norm independent of $D$. 
This result is an extension to Stone's seminal work from 1980 to the regime of high-dimensional target domain.
In addition, we unveil a connection between the failure probability $\eps$ and the number of samples required to achieve the optimal rates. 

%%%%%%%%%%%%%%%%%%%%%%%%%%%%%%%%%%%%%%
%%%  Baraks original version
%%%%%%%%%%%%%%%%%%%%%%%%%%%%%%%%%%%%%%
% Non-parametric estimation of functions as well as their derivatives by means of local-polynomial regression is a subject that was studied in the literature since the late 1970's. 
% Given a set of (noisy) function evaluations we perform a local polynomial fit, and by taking its $m$-th derivative we obtain an estimate for the $m$-th function derivative.
% The known optimal rates of convergence for this problem for a $k$-times smooth function operating from $\mathbb{R}^d$ to $\mathbb{R}$ are $n^{-\frac{k-m}{2k + d}}$.
% However in modern applications it is often the case that we have to estimate a function operating to $\RR^D$, for $D \gg d$ extremely large (e.g., high dimensional target domain).
% In this work, we prove that these same rates of convergence are also achievable by local-polynomial regression in case of a high dimensional target, given that the noise in the sample has a covariance matrix with operator norm bounded by $\sqrt{c/D}$ for some constant $c$.
% Noise models that respect this assumption include uniform distribution on a sphere or ball of radius $\sigma$, or normal distribution with covariance matrix of $\sqrt{c/D} I_D$, where $I_D$ is the $D$-dimensional identity matrix.
% This result is an extension to Stone's seminal work from 1980 to the regime of high-dimensional target domain.
% In addition, we unveil a connection between the failure probability $\eps$ and the number of samples required to achieve the optimal rates. 
\end{abstract}
\section{Introduction}

The problem of non-parametric estimation of functions is very well studied in the literature \cite{gyorfi2002distribution,ruppert1994multivariate,stone1980optimal, stone1982optimal,tsybakov2008introduction,wasserman2013all}.
The goal is to estimate a function $f:\RR^d\to\RR$ given a sample $(X, Y)$ drawn from some joint distribution, where $X\in \RR^d$ represents the domain of $f$ and $Y\in \RR$ are noisy evaluations of $f(x)$ at the sampled sites $\{x_i\}_{i=1}^n\subset \RR^d$.

One of the most common approaches to perform such an estimation is through Local Polynomial Regression \cite{cleveland1979robust,stone1977consistent}.
Explicitly, in order to estimate a $k$-times smooth function $f\in \C^k$ at some point $\xi$ we look for a polynomial $\pi^*\in \Pi_{k-1}^{d\to 1}$ of degree less than or equal to $(k-1)$ operating from a $d$-dimensional domain to $\RR$, and which solves the following minimization problem
\begin{equation}\label{eq:PI_star_1D}
   \pi^*_{\xi}(x) = \argmin_{\pi\in \Pi_{k-1}^{d\to1}} \frac{1}{\abs{I_{\xi,n}}}\sum_{i\in I_{\xi,n}} \abs{y_i - \pi(x_i - \xi)}^2
,\end{equation}
where $I_{\xi,n} = \{i ~|~ \norm{\xi-x_i}\leq \epsilon_n\}$, $\epsilon_n$ is a bandwidth parameter that will be discussed at more length later on, and $\abs{I_\xi}$ represents the number of elements in $I_\xi$.
Then, the estimation of $f(\xi)$ is given by $\pi^*_\xi(0)$.

This local polynomial mechanism described for function estimation is also used to estimate differential functionals \cite{ruppert1994multivariate,stone1980optimal,stone1982optimal}.
Let $ \alpha = (\alpha_1, \ldots, \alpha_d) $ be an integer valued multi-index, and define 
\[
\abs{\alpha} = \sum_{j=1}^d\alpha_j
,\]
\[
x^\alpha = x_1^{\alpha_1} x_2^{\alpha_2} \cdots x_d^{\alpha_d}
,\]
\[
\partial^\alpha = \partial^{\alpha_1}_{x^1}\cdots \partial^{\alpha_d}_{x^d}
.\]
And, define $L$ to be the following differential operator of degree $m$
\begin{equation}\label{eq:defL_1D}
    L = \sum_{\abs{\alpha}\leq m}c_\alpha \partial^\alpha 
.\end{equation}
Then the differential functional of degree $m$ defined by
\begin{equation}\label{eq:defT_1D}
    T_\xi(f) = L[f](\xi)
,\end{equation}
can be estimated based upon $n$ evaluations of the function $f$ by
\begin{equation}\label{eq:defTn_1D}
    T_\xi(f) \approx  \widehat T_{\xi,n} \defeq L[\pi^*_\xi](0).
\end{equation}

This estimator have been investigated by Stone in \cite{stone1980optimal} where he proved that it converges point-wise to the exact value at a rate of $\OO(n^{-\frac{k-m}{2k+d}})$. 
Stone gave a global convergence analysis as well to this problem in \cite{stone1982optimal}. 
Later, Ruppert and Wand showed explicit Mean Squared Error (MSE) analysis, with explicit optimal bandwidth computation for the case of function estimation \cite{ruppert1994multivariate}. In addition, they analyze there the convergence of the derivative estimation in the limited case of $d=1$.

However, in modern Big-data related problems it is often the case that functions have a high dimensional target.
Therefore, many works concern some type of regression in high dimensional space \cite{aizenbud2015OutOfSample,belkin2006manifold,fletcher2007riemannian,peterfreund2020local,salhov2020multi,sober2017approximation,steinke2010nonparametric}.
Furthermore, in 
\cite{aamari2017non,aizenbud2021non,boissonnat2014manifold,sober2016MMLS}, the estimation of low dimensional manifolds embedded in very high dimensional domain is performed by means of non-parametric estimation. 
%\mycomment{Other examples include \cite{bronstein2017geometric,belkin2006manifold,steinke2010nonparametric,fletcher2007riemannian}.} 
Thus, it raises the interest in cases where $f$ operates from $\RR^d$ to $\RR^D$, where $D \gg d$, and $L[f]$ of Eq. \eqref{eq:defT_1D}, receives a vector form.

The main result of this current work is the generalization of Stone's point-wise convergence analysis presented in \cite{stone1980optimal} to the vector-valued case.
Differently from Stone, for simplicity we limit the class of functions we analyze to $k$-times smooth functions.
Be that as it may, this restriction can be removed to handle the same class of functions handled by Stone without changing the proofs significantly.
This extension may seem at first glance as trivial by estimating each coordinate separately and using the union bound. 
However, since we assume that the target domain is of very large dimension, taking such a naive approach will result with bounds depending strongly on $D$.
Below we show that under certain noise assumptions (e.g., for noise uniformly distributed in a $\sigma$ ball, or a Gaussian noise with expected noise norm of $\sigma$) the convergence rates are maintained and there is no dependency in $D$.
Furthermore, in addition to the extension of Stone's result we unveil the relationship between the number of samples needed and the fail probability $\eps$. 
This is done through a generalization of the median trick~\cite{alon1999space} and may be of independent interest to the reader.

The convergence rates reported here for the vector-valued case aligns with the optimal rates of scalar-valued function estimation shown in \cite{stone1980optimal}.
In other words, our work shows achievability of these same rates for the vector-valued case.
Thus, as a by-product, we get that under our noise assumptions the optimal convergence rates are the same as the scalar-valued ones and we show that they are achievable by local polynomial regression.

\section{Algorithm description}
The algorithm we present here is a generalization of the local polynomial regression to vector valued functions. For any $\C^k$ function $f:\RR^d\to\RR^D$, and $L$ a differential operator of degree $m$ (see \eqref{eq:defL_1D}) we define the vector-valued functional $T$ by
\[
T(f) = L[f](0)
.\]
Note, that there is a slight discrepancy in notation from \eqref{eq:defT_1D}. 
This difference comes from the fact that we analyze the convergence point-wise, and thus we consider the operator evaluated only at zero (an exact analysis could be done near any point $\xi\in \RR^d$).

Our goal is to provide $\widehat T_n$ an estimate of $T(f)$ based on $n$ noisy samples of $f$ through applying $T$ on the local polynomial regression of $f$.
In other words, we define $\pi^*\in \Pi_{k-1}^{d\mapsto D}$ the vector valued polynomial regression of degree $k-1$ operating from $\RR^d$ to $\RR^D$ through the minimization problem
\begin{equation}\label{eq:pi_star_def}
    \pi^*(x) = \argmin_{\pi\in \Pi_{k-1}^{d\mapsto D}}\frac{1}{\abs{I_{n}}}\sum_{i\in I_{n}} \norm{y_i - \pi(x_i)}^2
,\end{equation}
where $I_n = \{i ~|~ \norm{x_i}\leq \epsilon_n\}$, and the bandwidth $\epsilon_n$ is chosen such that 
\begin{equation}\label{eq:epsilon_n_req}
0 < \lim\limits_{n\to\infty}n^{\frac{1}{2k+d}} \epsilon_n < \infty.    
\end{equation}
Then, the estimation of $T(f)$ is defined by
\begin{equation*}
    T(f) \approx \widehat T_n \defeq T(\pi^*)
.\end{equation*}
As shown in Theorem \ref{thm:approx} for any $\eps$ there are $N_\eps$ and $C_\eps$ large enough such that for all $n>N_\eps$ the error of this algorithm is bounded by $C_\eps n^{-\frac{k-m}{2k+d}}$ with probability of $1- \eps$.
The steps of this algorithm are summarized in Algorithm \ref{alg:vector_valued_estimation}.
\begin{algorithm}[H]
	\caption{Vector valued function estimation}
	\label{alg:vector_valued_estimation}
	\begin{algorithmic}[1]
		\State {\bfseries Input:}\begin{tabular}[t]{ll}
            $\{x_i\}_{i=1}^n\subset \RR^d $ & \\
            $\{y_i\}_{i=1}^n\subset \RR^D $ & \\
		\end{tabular}
		\State {\bfseries Output:}\begin{tabular}[t]{ll}
				$\widehat T_n \in \RR^D$ & Estimation of L[f](0).\\
		\end{tabular}
		\State Compute $\pi^*$ as in Eq. \eqref{eq:pi_star_def}
		\State Evaluate $T$ on $\pi^*$.
		\State\Return $\widehat T_n = T(\pi^*)$.
	\end{algorithmic}
\end{algorithm}

Unfortunately, Theorem \ref{thm:approx} does not give us direct connection between the failure probability $\varepsilon$ and $C_\varepsilon$, $N_\varepsilon$. For this reason we introduce a variant of the algorithm in Algorithm \ref{alg:vector_valued_median_trick}.
Theorem \ref{thm:N_C_bound} unveils the explicit connection between $\eps$, the failure probability of the algorithms and the required number of samples $N_\eps$ as well as the coefficient $C_\eps$.

\begin{algorithm}[H]
	\caption{Vector valued median trick}
	\label{alg:vector_valued_median_trick}
	\begin{algorithmic}[1]
		\State {\bfseries Input:}\begin{tabular}[t]{ll}
            $\{x_i\}_{i=1}^n\subset \RR^d $ & \\
            $\{y_i\}_{i=1}^n\subset \RR^D $ & \\
		\end{tabular}
		\State {\bfseries Output:}\begin{tabular}[t]{ll}
				$\widehat T_n \in \RR^D$ & Estimation of L[f](0).\\
		\end{tabular}
		\State Split the data into $\nu = \OO(\ln \frac{1}{\eps})$ equal parts.
		\State Apply Algorithm \ref{alg:vector_valued_estimation} on each part to get $T_{n/\nu}^1, \ldots T_{n/\nu}^\nu $.
		\State Find a ball of radius  $2 \frac{C_{\eps_0} \nu^r}{n^r}$ that has more than half of the estimates $\widehat T_{n/\nu}^i$ in it.
		\State\Return $\widehat T_n$ - the center of that ball.
	\end{algorithmic}
\end{algorithm}

\section{Main Result}
An important point in the high dimensional case is the noise distribution. That is, assuming that the noise is distributed in a ``cube"-like region, will cause the volume of the sampled region to grow exponentially with $D$. 
To avoid such dependency on the dimension $D$, we require that the noise vector $Z\in \RR^D$ will have a covariance matrix that satisfies $\|\cov(Z)\|_{op} \leq \sqrt{\frac{c}{D}}$. 
Examples for such distributions include a uniform distribution on a sphere or ball of some radius $\sigma$. 
Another example is a symmetric vector valued normal distribution with a bounded expected value of the norm. Such a normal distribution will has a covariance matrix of $\sqrt{\frac{c}{D}} I_D $. In general, any spherically symmetric  noise model, whose radial part have constant variance will fit our restrictions.

\subsection{Sampling assumptions and notations}\label{sec:sampling_assumptions}
We assume that there is a $\C^k$ function $f:\RR^d\to\RR^D$. Let $ (X,Y) $ be a pair of random variables such that $ X $ in $ \RR^d $ and $ Y $ in $ \RR^{D} $.
	Assume that the distribution of $ X $ is absolutely continuous and that its density is bounded away from zero and from infinity on $ U $, an open neighborhood of the origin of $ \RR^d $. We further assume that
	\[
	f(x) = \EE(Y ~|~ X=x).
	\]
Denote $Z = f(X)-Y$ to be the random noise. 
We assume that the noise covariance satisfies
\begin{equation}\label{eq:cov_Z_req}
    \|\cov(Z|X=x)\|_{op} = \|\cov(Y|X=x)\|_{op} \leq \sqrt{c/D},
\end{equation}
for some constant $c$.

Let $M$ be a matrix in $ \RR^{n\times m}$.
Drawing inspiration from programming languages, we denote the matrix $ij$-th entry by $M[i,j]$, the $i$-th row of the matrix by $M[i,:]$ and the $j$-th column by $M[:,j]$.
Furthermore, for convenience we refer to column vectors $v\in \RR^D$ as matrices of size $D\times 1$, and its $i$-th coordinate by $v[i]$.

Throughout the paper we denote the standard Euclidean norm by $\| \cdot \|$ and the Euclidean operator norm by $\| \cdot \|_{op}$.

\subsection{The main theorems}
We begin by showing, in Theorem \ref{thm:approx}, the convergence rates of Algorithm \ref{alg:vector_valued_estimation}.
Our proofs laid in Section \ref{sec:proof_thm_approx} follow the foot-steps of Stone's seminal work \cite{stone1980optimal} with the required adaptations for the high-dimensional case.

\begin{Theorem}\label{thm:approx}
	Let $f:\RR^d \to \RR^D, X, Y, Z$ be as defined in the sampling assumptions in Section~\ref{sec:sampling_assumptions}. Let $ L = \sum_{\abs{\alpha}\leq m}c_\alpha \partial^\alpha $, where $\alpha$ is an integer valued multi-index, and let $T$ be the functional
	\[
	T(f) = L[f](0)\in \RR^D
	.\]
	Given a sample $ \{(x_i, y_i)\}_{i=1}^{n} $ drawn i.i.d from $ (X, Y) $, we define the local polynomial regressor $ \pi^*  $  of Eq. \eqref{eq:pi_star_def},
	where $ \epsilon_n $ is such that Eq. \eqref{eq:epsilon_n_req} holds.
Furthermore, let us define the estimate
	\[
	\widehat T_{n} = L[\pi^*](0)\in \RR^D
	.\]
	Then, for any $\varepsilon$ there is C and $ N $ large enough such that for any $ n > N $ we have
	\[
	\Pr(\|{\widehat T_{n} - T(f)}\| > \frac{C}{n^r}) \leq \varepsilon
	,\]
	for $ r =\frac{k-m}{2k+d} $.
\end{Theorem}

Unfortunately, Theorem \ref{thm:approx} does not show an explicit connection between the failure probability $\varepsilon$ and $C_\varepsilon$, $N_\varepsilon$. For this reason we introduce the following theorem

\begin{Theorem}\label{thm:N_C_bound}
Under the assumptions of Theorem \ref{thm:approx}, Algorithm \ref{alg:vector_valued_median_trick} results in an estimator $\widehat{T}_n$ such that for any $ n > N_\eps $ we have
	\[
	\Pr(\|{\widehat T_{n} - T(f)}\| > \frac{C_\eps}{n^{r_m}}) \leq \varepsilon
	,\]
	for $ r_m =\frac{k-m}{2k+d} $, 
	\[
	N_\eps = N_0 \ln \frac{1}{\varepsilon},
	\]
	and 
	\[
	C_\eps = C_0 c_0^{r_m} \left(\ln \frac{1}{\varepsilon}\right)^{r_m},
	\]
	 where $C_0,c_0$ and $N_0$ are some constants.
\end{Theorem}

The proof is trivially followed from Theorem \ref{thm:approx}, and Lemma \ref{lem:vector_median_trick}, a variant of the "Median trick" \cite{alon1999space}.

\begin{Lemma}\label{lem:vector_median_trick}
    Let $T\in \RR^D$ be a value of interest that we are estimating using $n$ data points. Assume that we have an estimator $\widehat T_n $ that satisfies: For any $\eps>0$ there are $\bar{N}_\eps$ and $\bar{C}_\eps$  such that for any $n>\bar{N}_\eps$, 
    \[
	\Pr(\|{\widehat T_{n} - T}\| > \frac{\bar{C}_\eps}{n^r}) \leq \eps
	.\]
	Then, there are global constants $C_0, c_0$ and $N_0$, such that for any $\eps>0$ and for any $n>N_\eps$ Algorithm \ref{alg:vector_valued_median_trick} returns $\widehat T_{n}$
    \[
	\Pr(\|{\widehat T_{n} - T}\| > \frac{C_\eps}{n^r}) \leq \eps
	.\]
	where 
	\[
	N_\varepsilon = N_0 \ln \frac{1}{\varepsilon},
	\]
	\[
	C_\eps = C_0\left(c_0\ln \frac{1}{\varepsilon}\right)^r
	\]
\end{Lemma}
The proof of Lemma \ref{lem:vector_median_trick} is given in Section \ref{sec:proof_lem_median}

\section{Proofs}
For ease of notation we use the $\OO_p$ notation as defined below.
\begin{definition}\label{def:o_p}
% We say that $X_n = \OO(Y_n)$ if
% , for any $\eps>0$ there are $N,M>0$ such that 
% \[
% P(|X_n/Y_n| \geq M)<\eps \mbox{ for all } n>N
% \]
% In the case when $Y_n = 1$ we get that 
We say that $X_n = \OO_p(1)$ if
for any $\eps>0$ there are $N,M>0$ such that 
\[
P(|X_n| \geq M)<\eps \qquad\mbox{ for all } \qquad n>N.
\]
\end{definition}
\subsection{Proof of Theorem \ref{thm:approx}}
\label{sec:proof_thm_approx}

%\begin{proof}
Following the notation of Stone in \cite{stone1980optimal}, we define $I_n = \{i: 1\leq i\leq n ~,~\|x_i\|<\epsilon_n\}$,
\[
\delta_n = \max\{\|x_i\|: i\in I_n\}
,\]
and denote by $N_n$ the number of points in $I_n$. Since the PDF of $X$ is bounded away from 0 and infinity, we have that $n\epsilon_n^d/N_n = \OO_p(1)$. 
Thus, From Eq. \eqref{eq:epsilon_n_req} we get that $n^\frac{-2k}{2k+d}N_n = \OO_p(1) $ as well. Since $\frac{\delta_n}{\epsilon_n} = \OO_p(1)$ we have, similarly to equation (2.7) in \cite{stone1980optimal}, that 
\begin{equation}\label{eq:delta_k-m}
    \delta_n^{k-m} = n^{-r}\OO_p(1)
\end{equation}
and 
\begin{equation}\label{eq:N_inv_delta-2m}
    N_n^{-1}\delta_n^{-2m} = n^{-2r}\OO_p(1)
\end{equation}

Let $\alpha = (\alpha_1, \ldots, \alpha_d)$ be a non-negative integer multi-index, and denote by $A$ the set of all multi-indices
\[
A =\{\alpha = (\alpha_1, \ldots, \alpha_d) ~|~ \abs{\alpha} = \sum_{j=1}^d \alpha_j \leq k~,~ \alpha_j \geq 0\}
\]
Let $\XX_n \in \RR^{N_n \times |A|}$ be the matrix with entries
\[
\XX_n[i,\alpha] = \frac{x_i^\alpha}{\delta_n^{|\alpha|}} \in \RR
,\]
where for $v = (v_1, \ldots, v_d)^T\in \RR^d$
\[
v^\alpha =\prod\limits_{j=1}^d v_j^{\alpha_j}.
\]
Note, that 
\[
(\XX_n^T \XX_n)[\alpha,\beta] = \sum_{i\in I_n} \frac{x_i^\alpha x_i^\beta}{\delta_n^{|\alpha|} \delta_n^{|\beta|}} \qquad\qquad \alpha,\beta \in A
.\]
Following Eq. (2.9) of \cite{stone1980optimal}, we have that 
\begin{equation}\label{eq:N_xTX}
    N_n (\XX_n^T \XX_n)^{-1} = \OO_p(1).
\end{equation}
% or, for any $\eps$ there is $M_\eps>0$ and $N_\eps$ such that, \mycomment{Do we really need the interpretation of $\OO_p$ here?}
% \begin{equation}\label{eq:N_xTX_2}
%     P(|( N_n (\XX_n^T \XX_n)^{-1})_{ij}| \geq M_{\eps}) \leq \eps \quad \forall n> N_{\eps}
% \end{equation}

Define by $f_k$ the vector valued $(k-1)$ degree Taylor approximation of $f$ at $0$.
Let $\YY_n\in \RR^{N_n \times D}$ be the samples of $f(x)$ corresponding to the sites $x_i$ for $i\in I_n$.
Define $\JJ_n,\JJ_{kn},\ZZZ_n \in \RR^{N_n\times D}$ to be the samples of $f(x), f_k(x)$, and the stochastic sample errors correspondingly; i.e.,
\begin{align*}
    \JJ_n[i,:] &= f(x_i)^T\\
    \JJ_{kn}[i,:] &= f_k(x_i)^T \\
    \ZZZ_n &= \YY_n-\JJ_n.
\end{align*}
We denote by $\ZZZ[i,j]$ the $i,j$th entry of $\ZZZ_n$. Furthermore, we note that each row of $\ZZZ_n$ (i.e. $\ZZZ[i,:]$) is distributed like $Z = Y-f(X)$ and from Eq. \eqref{eq:cov_Z_req} we have that 
\begin{equation}\label{cov_Z_bound}
    \|\cov(\ZZZ[i,:]|\XX) \| \leq \sqrt{c/D}
.\end{equation}
Note, that 
\[
\JJ_{kn}[i,j] =  \sum_{|\alpha|\leq k}\frac{x_i^{\alpha} \partial^\alpha f^{(j)}(0) } {\abs{\alpha}!}  
=
\sum_{|\alpha|\leq k}    
\left(
\frac{x_i^{\alpha}} {\delta_n^{|\alpha|}}
\right) \frac{\delta_n^{|\alpha|} \partial^\alpha f^{(j)}(0) } {\abs{\alpha}!}  \qquad \qquad i \in I_n
,\]
where $f(x) = (f^{(1)}(x), \ldots, f^{(D)}(x))^T$.
From the fact that $\JJ_{kn}$ are samples of the Taylor expansion of $f(x)$ we get that
\[
((\XX_n^T \XX_n)^{-1} \XX_n^T \JJ_{kn})[\alpha,j] = \frac{\delta_n^{|\alpha|} } {\abs{\alpha}!} \partial^\alpha f^{(j)}(0)
,\]
since the Taylor polynomial minimizes the Least-Squares error.

Accordingly, we define $\LLL_n \in \RR^{|A|}$ by
\begin{equation}\label{eq:L_def}
\LLL_n[\alpha] = \frac{\abs{\alpha}! c_\alpha}{\delta_n^{|\alpha|}}
,    
\end{equation}
where 
$c_\alpha = 0$ for  $ m < |\alpha| \leq k$.
Following these definitions, we have 
\[
T(f)^T = \LLL_n^T (\XX_n^T \XX_n)^{-1} \XX_n^T \JJ_{kn} \in \RR^{1\times D}.
\]
In addition, from the definition of $\pi^*$ in \eqref{eq:pi_star_def} we know that its coefficients in the basis $\{\frac{x^{\alpha}}{\delta_n}\}_{\alpha\in A}$ solve the following Least-Squares equations
\[
\vec \pi^*_c = (\XX_n^T \XX_n)^{-1} \XX_n^T \YY_{n} \in \RR^{|A|}
.\]
Thus, we have that 
\[
\widehat{T}_n^T = \LLL_n^T \vec \pi^*_c = \LLL_n^T (\XX_n^T \XX_n)^{-1} \XX_n^T \YY_{n} \in \RR^{1\times D}
.\]
Since $\YY_n = \ZZZ_n + \JJ_n$, we have,
\begin{equation*}%\label{eq:err_partition}
\widehat{T}_n^T - T(f)^T = \underbrace{\LLL_n^T (\XX_n^T \XX_n)^{-1} \XX_n^T \ZZZ_{n}}_{\mbox{stochastic error}} + \underbrace{\LLL_n^T (\XX_n^T \XX_n)^{-1} \XX_n^T (\JJ_{n}-\JJ_{kn})}_{\mbox{Taylor estimation error}}
\end{equation*}
and by the triangle inequality, we get 
\begin{equation}\label{eq:err_partition}
\|\widehat{T}_n^T - T(f)^T\| \leq \underbrace{\|\LLL_n^T (\XX_n^T \XX_n)^{-1} \XX_n^T \ZZZ_{n}\|}_{\operatorname{Err}_S} + \underbrace{\|\LLL_n^T (\XX_n^T \XX_n)^{-1} \XX_n^T (\JJ_{n}-\JJ_{kn})\|}_{\operatorname{Err}_T}
\end{equation}
From Lemma \ref{lem:DTaylor} we have that 
\[
\|\JJ_n[i,:] - \JJ_{kn}[i,:]\| \leq M\delta_n^k
\]
where $M$ is some constant.
Thus, since $\|x_i\|\leq \delta_n$, we have 
\[\|\XX_n^T (\JJ_{n}-\JJ_{kn})[\alpha,:]\|
=
\|\sum_{i=1}^{N_n} \XX_n^T[\alpha,i] (\JJ_{n}-\JJ_{kn})[i,:] \|
\leq M \sum_{i\in I_n} \frac{\|x_i\|^{|\alpha|}}{\delta_n^{|\alpha|}} \delta_n^k \leq M N_n \delta_n^k \qquad \qquad \mbox{for }\alpha \in A
\]
Therefore, from \eqref{eq:delta_k-m} and \eqref{eq:N_xTX} we have that 
\begin{align}\label{eq:bound_taylor_err}
\operatorname{Err}_T &= \|\LLL_n^T (\XX_n^T \XX_n)^{-1} \XX_n^T (\JJ_{n}-\JJ_{kn})\| \notag\\
&\leq \|\LLL_n^T (\XX_n^T \XX_n)^{-1}\| \cdot \| \XX_n^T (\JJ_{n}-\JJ_{kn})\|\notag\\
& \leq \|\LLL_n^T (\XX_n^T \XX_n)^{-1}\| \cdot \|\mathbbm{1} M N_n \delta_n^k\|\notag\\
& = \|\LLL_n^T N_n(\XX_n^T \XX_n)^{-1}\| \cdot \|\mathbbm{1} M \delta_n^k\|\notag\\
&= n^{-r}\OO_p(1),
\end{align}
Where $\mathbbm{1}$ is a vector of all ones.
Now we concentrate on the stochastic part of Eq. \eqref{eq:err_partition}, namely $\operatorname{Err}_S$. 
Note, that from the linearity of the expected value we get
\begin{equation}\label{eq:bound_E}
\EE(\LLL_n^T (\XX_n^T \XX_n)^{-1} \XX_n^T \ZZZ_{n} | \XX_n) = \vec{0} \in \RR^{1\times D}
.\end{equation}
So, if we bound 
\[
\operatorname{Cov}(\LLL_n^T (\XX_n^T \XX_n)^{-1} \XX_n^T \ZZZ_{n} | \XX_n)
\]
we could use Chebyshev inequality.
Since $\LLL_n^T (\XX_n^T \XX_n)^{-1} \XX_n^T \ZZZ_{n} \in \RR^{1\times D}$, we have from Eq. \eqref{eq:bound_E},
\begin{align} \label{eq:cov_expressionF}
\operatorname{Cov}(\LLL_n^T (\XX_n^T \XX_n)^{-1} \XX_n^T \ZZZ_{n} | \XX_n) &= \EE((\LLL_n^T (\XX_n^T \XX_n)^{-1} \XX_n^T \ZZZ_{n})^T \cdot \LLL_n^T (\XX_n^T \XX_n)^{-1} \XX_n^T \ZZZ_{n} | \XX_n) \notag\\
& = \EE(\ZZZ_{n}^T F \ZZZ_{n} | \XX_n) 
,\end{align}
where we define,
\[
F = (\LLL_n^T (\XX_n^T \XX_n)^{-1} \XX_n^T )^T \cdot \LLL_n^T (\XX_n^T \XX_n)^{-1} \XX_n^T
.\]
Accordingly, although $F\in \RR^{N_n\times N_n}$, it is a rank-1 matrix, and thus, it can be written as 
\[
F = \xi vv^T,
\]
where $\xi$ is the (non-zero) eigenvalue of $F$ and $v \in \RR^D$ with $\| v \| = 1$.
Since for any matrix $A$ a non-zero eigenvalue of $A^TA$ is also an eigenvalue of $AA^T$, we get
\begin{align}
 \xi = \LLL_n^T (\XX_n^T \XX_n)^{-1} \XX_n^T \cdot (\LLL_n^T (\XX_n^T \XX_n)^{-1} \XX_n^T )^T &= \LLL_n^T (\XX_n^T \XX_n)^{-1} \XX_n^T \cdot \XX_n (\XX_n^T \XX_n)^{-1}  \LLL_n \notag \\
 = \LLL_n^T (\XX_n^T \XX_n)^{-1}  \LLL_n
\end{align}
Furthermore, we also know the following facts:
\begin{enumerate}
    \item The entries of $(\XX_n^T \XX_n)^{-1}$ are $\OO_p(1)/N_n$ (see Eq. \eqref{eq:N_xTX})
    \item $\max_j(\LLL_n[j]) = \delta_n^{-m} \OO(1)$ (see Eq. \eqref{eq:L_def})
    \item $N_n^{-1}\delta_n^{-2m} = n^{-2r}\OO_p(1)$ (see Eq. \eqref{eq:N_inv_delta-2m})
\end{enumerate}
Thus, we achieve that 
\[
\xi = n^{-2r}\OO_p(1)
.\]
% Alternatively, for any $\eps$ there is $M_\eps>0$ and $N_\eps$ such that,
% \begin{equation}\label{eq:xi_bound}
%     P(\xi \geq M_{\eps} n^{-2r} ) \leq \eps \quad \forall n> N_{\eps}
% \end{equation}
% \mycomment{Again, do we need this explicitly? or could we just use the $\OO_p$}

Thus, we can rewrite Eq. \eqref{eq:cov_expressionF} as
\[
\operatorname{Cov}(\LLL_n^T (\XX_n^T \XX_n)^{-1} \XX_n^T \ZZZ_{n} | \XX_n) = \xi \EE(\ZZZ_n^Tvv^T\ZZZ_n| \XX_n) = \xi \EE( ww^T| \XX_n)
,\]
where 
\[
w^T = v^T \ZZZ =  \left(
\begin{array}{c}
     \sum_1^N v_i \ZZZ[i,1] \\
     \vdots\\
     \sum_1^N v_i \ZZZ[i,D] 
\end{array}
\right)^T\in \RR^{1\times D}
.\]
Thus, we have
\begin{align}
\EE(w w^T| \XX_n)[k,l] &= \EE\left(\sum_{a=1}^N v_a^2 \ZZZ[a,k] \ZZZ[a,l] + \sum_a\sum_{b\neq a} v_a v_b \ZZZ[a,k] \ZZZ[b,l] \bigg{|} \XX_n\right)\notag
\\
(\small{ \ZZZ[a,:] \mbox{ is indep. of } \ZZZ[b,:]})~~~~~~~&=
\EE\left(\sum_{a=1}^N v_a^2 \ZZZ[a,k]  \ZZZ[a,l]\bigg{|} \XX_n\right)\notag
\\
&=
\sum_{a=1}^N v_a^2 \EE(\ZZZ[a,k] \ZZZ[a,l] | \XX_n)\notag
\\
(\small{ \ZZZ[a,:] \mbox{ are identically dist.}})~~~~~~~&=
\EE(\ZZZ[1,k] \ZZZ[1,l]| \XX_n)\sum_{a=1}^N v_a^2\notag
\\
(\small{ \|v\| = 1})~~~~~~~&=
\EE(\ZZZ[1,k] \ZZZ[1,l]| \XX_n)  = \cov(Z| X     )[k,l]
% \\
% &=\left\{ \begin{array}{ll}
%      \EE(\sum_{i=1}^N v_i^2 z_{ik}^2 + \sum_a \sum_{b\neq a} v_av_b z_{ak}z_{bk}) &  k=j \\
%      \EE(\sum_{i=1}^N v_i^2 z_{ik} z_{kj} + \sum_a \sum_{b\neq a} v_a v_b z_{ak} z_{bj})& k \neq j
% \end{array} \right. \notag \\
% \mbox{\small{(noise in different samples is indep.)}}&= 
%  \left\{ \begin{array}{ll}
%      \sum_{i=1}^N v_i^2 \EE(z_{ik}^2) &  k=j \\
%      \sum_{i=1}^N v_i^2 \EE(z_{ik} z_{kj}) & k \neq j
% \end{array} \right.
% \\
% \mbox{\small{(noise in different samples is indep.)}}&= 
%  \left\{ \begin{array}{ll}
%      \sum_{i=1}^N v_i^2 \EE(z_{ik}^2) &  k=j \\
%      v_k^2 \EE(z_{kk} z_{kj}) & k \neq j
% \end{array} \right.
% \\
% &= 
% \left\{ \begin{array}{ll}
%      \EE(z_{1k}^2) &  k=j \\
%      0 & k \neq j
% \end{array} \right.
% \\
% &= \cov(Z)  ~~~~~~~~ \mbox{ Where} Z \mbox{ is the noise vector r.v.}
\end{align}
and so $\EE(w w^T) = \cov(Z)$
and, 
\[
\cov(\LLL_n^T (\XX_n^T \XX_n)^{-1} \XX_n^T \ZZZ_{n} | \XX_n) = \xi \cov(Z) %= \xi \frac{c}{D}I
.\]
Using the vector valued Chebyshev inequality, denoting 
\[
G = (\LLL_n^T (\XX_n^T \XX_n)^{-1} \XX_n^T \ZZZ_{n})^T \in \RR^D
,\]
we have
\[
\Pr \left( \sqrt{G^T\frac{1}{\xi}\cov(Z)^{-1}G} > t | \XX_n\right) \leq \frac{D}{t^2}
.\]
Note, that for any $v \in \RR^D$ and positive definite $A\in \RR^{D\times D}$ with largest eigenvalue $\sigma_1(A)$, 
\[
\sqrt{v^TA^{-1}v} \geq \frac{1}{\sigma_1(A)} \|v\|
.\]
From this, and from \eqref{cov_Z_bound} ( $\sigma_1(\cov(Z)) = \|\cov(Z)\|_{op} =  \sqrt{\frac{c}{D}}$) follows that 
\[
\Pr ( \|g\| > \frac{t \sqrt{\xi c}}{\sqrt{D}} | \XX_n) \leq \frac{D}{t^2}
\]
Taking $t' = \frac{t\sqrt{c} }{\sqrt{D}} $, we have that  $t = \frac{\sqrt{D} t'}{\sqrt{c}} $, and 
\[
\Pr ( \|g\| > t' \sqrt{\xi} | \XX_n) \leq \frac{c}{t'^2}
.\]
Alternatively, for any $\eps$ there is $M_\eps>0$ and $N_\eps$ such that,
\begin{equation}\label{eq:bound_stoc_err_vector}
\Pr ( \|\LLL_n^T (\XX_n^T \XX_n)^{-1} \XX_n^T \ZZZ_{n}\| > M_\eps n^{-r} | \XX_n) \leq \eps \quad \forall n> N_{\eps}
\end{equation}
Combining Eqs. \eqref{eq:err_partition}, \eqref{eq:bound_taylor_err}, and \eqref{eq:bound_stoc_err_vector}, we have 
\[
\|\widehat{T}_n^T - T(f)^T\| = n^{-r}\OO_p(1),
\]
and conclude the proof.
%\end{proof}

\begin{Lemma}[vector valued Taylor] \label{lem:DTaylor}
Let $f:\RR^d \to \RR^D$ be a $\mathcal{C}^k$ function. Let $f_k$ be the vector valued $(k-1)$ degree Taylor approximation of $f$ at $0$. Then, there is a constant $M$, such that for any point $x\in \RR^d$, with $\|x\| \leq \delta$ we have,
\[
\|f(x)-f_k(x)\|\leq M\delta^k
\]
\end{Lemma}
\begin{proof}
for any vector $v\in \RR^D$ with $\|v\| = 1$ we define $f,f_k:\RR^d \to \RR$, as follows:  $f^v(x) = v^T\cdot f(x)$ and $f^v_k(x) = v^T\cdot f_k(x)$.
Note that from the linearity of the Taylor series operator (i.e., the operator that given a function yields its Taylor series), we have that $f^v_k$ is the Taylor approximation of $f^v$, and thus we have that
\[
|f^v(x) - f_k^v(x)| \leq M\delta^k
,\]
where
\[
M = \max_{\|v\|=1, \abs{\alpha}= k, \norm{x}\leq \delta} \abs{\partial^\alpha f^v(x)}
.\]
Since $|f^v(x) - f_k^v(x)| = |v^T \cdot (f(x) - f_k(x))|$ and since this argument holds for $v = \frac{f(x) - f_k(x)}{\|f^v(x) - f_k^v(x)\|}$ we conclude the proof.
\end{proof}

\subsection{Proof of Lemma \ref{lem:vector_median_trick}} \label{sec:proof_lem_median}

%\begin{proof}
This proof follows the idea of the ``median trick" \cite{alon1999space} for high dimensional data.
%Let $N_0$ and $C_0$ be the constants of Theorem \ref{thm:approx} for the choise . 
Let $0.5>\eps_0>0$ to be set later. Let $\nu$ be some integer to be set later. Set $N_\eps = \nu \bar{N}_{\eps_0}$. Assuming $n>N_\eps$,
we divide our $n$ samples into $\nu$ equal parts, and preform the estimation of $T$ $\nu$ times to get $T_{n/\nu}^1, \ldots T_{n/\nu}^\nu $ such that 
\begin{equation}\label{eq:stone_for_nu_subevents}
\Pr(\|{\widehat T_{n/\nu}^i - T}\| > \frac{\bar{C}_{\eps_0} \nu^r}{n^r}) \leq \eps_0 \qquad \mbox{ for } 1\leq i \leq \nu .
\end{equation}

Using Hoeffding inequality \cite{hoeffding1994probability}, we have that the event in \eqref{eq:stone_for_nu_subevents} holds for at least $(\eps_0-a)\nu$ times (out of $\nu$) with probability of at least $1- exp(-2a^2\nu)$. Thus, for 
\[
\nu = \frac{\ln(\frac{1}{\eps})}{2(0.5-\eps_0)^2}
\]
we have that at least half of the events in \eqref{eq:stone_for_nu_subevents} hold with probability of at least $1-\eps$. In case that at least half of the events in \eqref{eq:stone_for_nu_subevents} hold, we have that any ball $B$ of radius  $2 \frac{C_{\eps_0} \nu^r}{n^r}$ that has more than half of the estimates $\widehat T_{n/\nu}^i$ in it, will also have $ T(f)$ in it. We also know that there is at least one such ball (the one around $T(f)$). Thus, denoting the center of such a ball $B$ by  $\widehat T_{n}$, we have that 
\begin{equation*}
\Pr(\|{\widehat T_{n} - T(f)}\| > 2\frac{C_{\eps_0} \nu^r}{n^r}) \leq \eps
\end{equation*}
or, substituting $\nu$ from above, we have 
\begin{equation*}
\Pr(\|{\widehat T_{n} - T(f)}\| > 2C_{\varepsilon_0}\left(\frac{\ln \frac{1}{\varepsilon}}{2(0.5-\varepsilon_0)^2}\right)^r \frac{1} {n^r}) \leq \eps
\end{equation*}
The required number of points for this algorithm to work is such that $n/\nu \geq \bar{N}_{\eps_0}$, or , in other words,
\[
	N_\varepsilon = \bar{N}_{\varepsilon_0}\frac{\ln \frac{1}{\varepsilon}}{2(0.5-\varepsilon_0)^2}.
\]
Denoting,
\[
	N_0 = \frac{\bar{N}_{\varepsilon_0}}{2(0.5-\varepsilon_0)^2},
\]
and 
\[
C_0 = 2C_{\eps_0} \qquad\qquad c_0 =  \frac{1}{2(0.5-\varepsilon_0)^2}
\]
we conclude the proof by setting $\eps_0 = 0.4$.
%\end{proof}

\section{Numerical Simulations}
We demonstrate the rates of convergence demonstrated in this paper with the following example.
We estimated a function $f:\RR \to \RR^D$ based on different number of samples, and for different values of $D$. Each of the $D$ coordinates of the function $f$ is a random polynomial of degree two. We ran the experiment for $D = 1,2,10,100,1000$ for number of samples $n$ ranges from $10^2$ to $10^5$. Each experiment was repeated $50$ times. The average over the $50$ experiments, along with the variance are shown in Figure \ref{fig:res}. The figure illustrates that indeed, the rates of convergence are independent of the dimensions $D$. The coefficient $c$ before the asymptotic rate in this example is $c\approx 0.096$. The bandwidth chosen in this experiment is $\epsilon_n =  n^{\frac{-1}{2k+d}}$. Simulation code is available at \textcolor{blue}{\href{https://github.com/aizeny/vector-valued-function-estimation}{https://github.com/aizeny/vector-valued-function-estimation}}.

\begin{figure}
	\centering
	\includegraphics[width=0.6\textwidth]{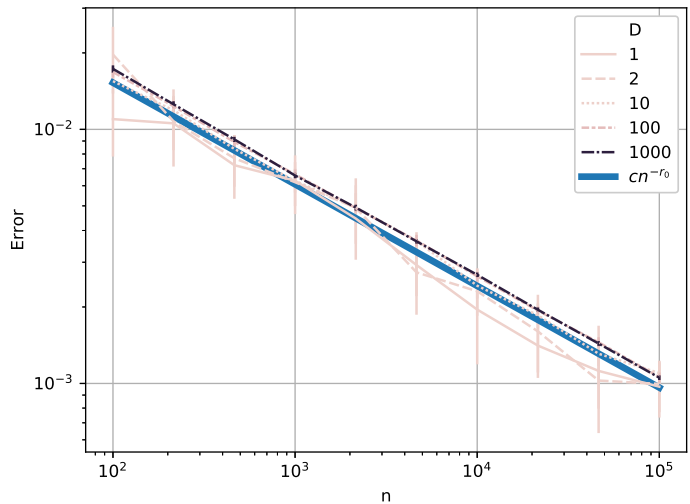}
	\caption{Estimating $f$ for different $n$ (x-axis) and different $D$ (different line color).}
	\label{fig:res}
\end{figure}

\section{Acknowledgments}
We wish to thank Prof. Felix Abramovich for driving us to do this work and for the fruitful discussions.
B. Sober is supported by Duke University and the Simons Foundation through Math+X grant 400837. 
%Y. Aizenbud is supported by Yale University.

\bibliography{main}{}
\bibliographystyle{plain}

\end{document}